\documentclass[11pt, reqno]{amsart}
\usepackage{eucal}

\usepackage{graphicx} 
\usepackage{comment}
\usepackage{color}
\usepackage{amsmath}
\numberwithin{equation}{section}

\usepackage[colorlinks, linkcolor=black, citecolor=blue, linktocpage,urlcolor=blue]{hyperref}
\theoremstyle{definition}
\newtheorem{thm}{Theorem}[section]
\newtheorem{definition}[thm]{Definition}
\newtheorem{proposition}[thm]{Proposition}

\newtheorem{lemma}[thm]{Lemma}
\newtheorem{corollary}[thm]{Corollary}

\newtheorem{remark}[thm]{Remark}

\author{Haiping Fu}
\address{Department of Mathematics, Nanchang University, Nanchang 330031, People’s Republic of China}
\email{mathfu@126.com}
\thanks{Supported in part by National Natural Science Foundations of China \#12461008 and 12271069,  Jiangxi Province
	Natural Science Foundation of China \#20202ACB201001, Jiangxi Province Graduate Student Innovation Special Fund Project \#YC2025-B037.}
\author{Yao Lu}
\address{Department of Mathematics, Nanchang University, Nanchang 330031, People’s Republic of China}
\email{luyao@email.ncu.edu.cn}

\usepackage{fancyhdr}


\makeatletter
\@namedef{subjclassname@2020}{
	\textup{2020} Mathematics Subject Classification}
\makeatother

\begin{document}
	\newcommand{\Ext}{\bigwedge\nolimits}
	\newcommand{\Div}{\operatorname{div}}
	\newcommand{\Hol} {\operatorname{Hol}}
	\newcommand{\diam} {\operatorname{diam}}
	\newcommand{\Scal} {\operatorname{Scal}}
	\newcommand{\scal} {\operatorname{scal}}
	\newcommand{\Ric} {\operatorname{Ric}}
	\newcommand{\Hess} {\operatorname{Hess}}
	\newcommand{\grad} {\operatorname{grad}}
	\newcommand{\Rm} {\operatorname{Rm}}
	\newcommand{ \Rmzero } {\mathring{\Rm}}
	\newcommand{\Rc} {\operatorname{Rc}}
	\newcommand{\Curv} {S_{B}^{2}\left( \mathfrak{so}(n) \right) }
	\newcommand{ \tr } {\operatorname{tr}}
	\newcommand{ \Riczero } {\mathring{\Ric}}
	\newcommand{ \Ad } {\operatorname{Ad}}
	\newcommand{ \dist } {\operatorname{dist}}
	\newcommand{ \rank } {\operatorname{rank}}
	\newcommand{\Vol}{\operatorname{Vol}}
	\newcommand{\dVol}{\operatorname{dVol}}
	\newcommand{ \zitieren }[1]{ \hspace{-3mm} \cite{#1}}
	\newcommand{ \pr }{\operatorname{pr}}
	\newcommand{\diag}{\operatorname{diag}}
	\newcommand{\Lagr}{\mathcal{L}}
	\newcommand{\av}{\operatorname{av}}
	\newcommand{ \floor }[1]{ \lfloor #1 \rfloor }
	\newcommand{ \ceil }[1]{ \lceil #1 \rceil }
	\newcommand{\Sym} {\operatorname{Sym}}
	\newcommand{\bcirc}{ \ \bar{\circ} \ }
	\newcommand{\sign}[1]{\operatorname{sign}(#1)}
	\newcommand{\cone}{\operatorname{cone}}
	\newcommand{\pbd}{\varphi_{bar}^{\delta}}
	\newcommand{\End}{\operatorname{End}}
	
	\renewcommand{\labelenumi}{(\alph{enumi})}
	\newtheorem{maintheorem}{Theorem}[]
	\renewcommand*{\themaintheorem}{\Alph{maintheorem}}
	\newtheorem*{remark*}{Remark}

	\vspace*{-1cm}

	\title{New rigidity theorem of Einstein manifolds and curvature operator of the second kind}
	\begin{abstract}
		Using Bochner techniques, we prove that an  Einstein manifold of dimension $n \ge 4$ has constant curvature provided that the curvature operator of the second kind satisfies a cone condition that is strictly weaker than nonnegativity. Furthermore, employing a result of Li \cite{Li5}, we establish that any closed Einstein manifold of dimension $n \ge 4$ satisfying
		\[k^{-1}({\lambda }_1+\cdots +{\lambda }_k)\ge -\theta(n,k) \bar{\lambda },\quad \text{for some} \quad k  \le [\frac{n+2}{4}]\]
		must be either flat or a spherical space form. Here, ${\lambda }_1\le {\lambda }_2\le \cdots \le {\lambda }_{\frac{(n-1)(n+2)}{2}}$
		are the eigenvalues of $\mathring{R}\,$, $\bar{\lambda }$ is their average, and $\theta (n,k)$ is a positive constant. This result generalizes the work of Dai-Fu \cite{DF} and Chen-Wang \cite{CW1,CW}. We also classify four-dimensional Einstein manifolds satisfying a cone condition.
	\end{abstract}
	\maketitle
	
	\pagestyle{fancy}
	\fancyhead{} 
	\fancyhead[CE]{New rigidity theorem of Einstein manifolds and curvature operator of the second kind}
	\fancyhead[CO]{Haiping Fu and Yao Lu}
	\fancyfoot{}
	\fancyfoot[CE,CO]{\thepage}
	\section{introduction}\hspace*{5mm}
	
	The study of the Riemannian curvature operator of the second kind originated with Nishikawa’s conjecture \cite{Ni}, which states that a closed simply connected Riemannian manifold with a nonnegative curvature operator of the second kind is diffeomorphic to a locally symmetric space, and if the operator is positive, the manifold is diffeomorphic to a spherical space form. This conjecture was confirmed by Cao–Gursky–Tran \cite{CGT} under the assumption of two-positivity, and later strengthened by Li \cite{Li4} and Nienhaus–Petersen–Wink \cite{NP} under the weaker condition of three-nonnegativity. Further classification results under various curvature conditions of the second kind can be found in \cite{DFD,Li3,Li7,Li6,Li4,Li2,Li5,DF,NP}.
	
	In the context of Einstein manifolds, Kashiwada \cite{KT} showed that closed Einstein manifolds with a positive curvature operator of the second kind are constant curvature spaces. Li \cite{Li3} later proved that Einstein manifolds of dimension $n\ge4$ with $4\frac{1}{2}$-positive (resp.$4\frac{1}{2}$-nonnegative) curvature operator of the second kind are constant curvature spaces (resp. locally symmetric), using Brendle’s result \cite{Brendle}. Nienhaus, Petersen and Wink \cite{NP,NPWW} used the Bochner formula to show that $n$-dimensional compact Einstein manifolds with $k(<\frac{3n(n+2)}{2(n+4)})$-nonnegative curvature operators of the second kind are either rational homology spheres or flat. Dai-Fu \cite{DF} first established a novel Bochner formula on Einstein manifolds. Based on this, they obtained that a compact Einstein manifold must have constant curvature if its curvature operator of the second kind is $k$-nonnegative, where $k=1$ for $4\leq n\leq7$, $k=2$ for $8\leq n\leq10$ and  $k=[\frac{n+2}{4}]$ for $n\geq11$. The Bochner formula can also be applied to compact Riemannian manifolds with $k$-positive curvature operator of the first kind.  Recent progress in this direction is documented in \cite{CMR,PW,BG,YZ,HZ}.
	
	Recently, Li \cite{Li5} studied cone conditions of the curvature operator of the second kind, which satisfies
	$${\alpha }^{-1}({\lambda }_1+\cdots +{\lambda }_{\alpha })\ge -\theta \bar{\lambda },$$
	for  the first $\alpha $ smallest eigenvalues of the curvature operator of the second kind. This condition is denote by $C(\alpha ,\theta )$ and $\mathring{C}\,(\alpha ,\theta )$ if the inequality is strict. Li proved that if $\mathring{R}\,\in C(\frac{n+2}{2},\theta )$ for some $-1<\theta <\frac{2(n-1)}{n+2},$ then $M$ is either flat or a rational homology sphere.
	
	Building on the work of Li \cite{Li5} and Dai–Fu \cite{DF}, Cheng and Wang \cite{CW} relaxed the curvature assumption to a cone condition. They introduced a positive constant $\theta(n)$ and proved that if an $n(\ge8)$-dimensional closed Einstein manifold satisfies $\lambda_1+\lambda_2\ge-2\theta(n)\bar{\lambda}$, then it is either flat or a spherical space form. This condition is strictly weaker than $2$-nonnegativity, which corresponds to $\theta=0$.
	
	In this paper, we investigate Einstein manifolds of dimension $n\ge4$ whose curvature operator of the second kind satisfies a cone condition that is strictly weaker than nonnegativity. Let $\{\lambda_{i}\}_{i=1}^{N}$ be the eigenvalues of $\mathring{R}\,$ with arranged in non-decreasing order, and $\bar{\lambda}=\sum\nolimits_{i=1}^{N}{{\lambda }_i}/{N}\;$ their average. By applying results of Dai–Fu \cite{DF} and Li \cite{Li5}, we obtain the following theorems.
	
	\begin{maintheorem}\label{thm:main-theorem-1}
		Let $(M^n, g)$ be an $n$-dimensional  Einstein manifold. For $n \ge 4,n\ne 6, 7, 10$ and $k\le[\frac{n+2}{4}]$, there exists a constant $\theta(n,k) >0$ such that if the curvature operator of the second kind satisfies 
		\[k^{-1}({\lambda }_1+\cdots +{\lambda }_k)\ge -\theta(n,k) \bar{\lambda },\]
		where $$\theta(n,k) =\frac{(N-k)(2N-9n+6)-3nN(k-2)}{(N-3)(N-k)+3kn(N-2)},N=\frac{(n-1)(n+2)}{2},$$
		then $M$ is either flat or a spherical space form.
	\end{maintheorem}
	Theorem ~\ref{thm:main-theorem-1}  admits an enhancement for dimensions $n = 4$ and $5$, as stated below.
		\begin{maintheorem}\label{thm:main-theorem-2} 
		Let $(M, g)$ be an $n = 4, 5$-dimensional  Einstein manifold, and $\mathring{R}$ be the curvature operator of the second kind. If\\
		(i) $n=4$ and $\mathring{R}\,$ satisfies ${\lambda }_1\ge -\frac{1}{7}\bar{\lambda }$,\\
		(ii) $n=5$ and $\mathring{R}\,$ satisfies ${\lambda }_1\ge -\frac{4}{9}\bar{\lambda }$,\\
		then $M$ is either flat or a spherical space form.
	\end{maintheorem}	
	\begin{remark}
		For the case $n=4$, Li established a more general cone condition covering all $1\le k < 9$; see \cite[Throrem 1.5]{Li5}. Although this condition is stronger than that of Li, we handle it using the Bochner technique. For completeness, we include the details here. As will be seen from the subsequent Theorem D and Corollary 4.1, our approach achieves the optimal result.
	\end{remark}
	\begin{remark}
		For dimensions $n=6,7$, we have $[\frac{n+2}{4}]=2$, but our method does not yield a corresponding cone condition for $k=2$. Similarly, for $n=10$, the method does not provide a cone condition for $k=3$. Nevertheless, the following theorem remains valid for  $n=6,7$ with $k=1$, and for $n=10$ with $k=1,2$.
	\end{remark}
	\begin{maintheorem}\label{thm:main-theorem-3} 
		Let $(M, g)$ be an $n = 6, 7$ or $10$-dimensional  Einstein manifold, and $\mathring{R}$ be the curvature operator of the second kind. If\\
		(i) $n=6$ and $\mathring{R}\,$ satisfies ${\lambda }_1\ge -\frac{208}{647}\bar{\lambda }$, \\
		(ii) $n=7$ and $\mathring{R}\,$ satisfies ${\lambda }_1\ge -\frac{163}{383}\bar{\lambda }$, or\\
		(iii) $n=10$ and $\mathring{R}\,$ satisfies ${\lambda }_1\ge -\frac{964}{1421}\bar{\lambda }$ or ${\lambda }_1+{\lambda }_2\ge -\frac{16}{37}\bar{\lambda},$\\
		then $M$ is either flat or a spherical space form.
	\end{maintheorem}
	\begin{remark}
		When $\theta=0$, Theorem A,B,C are the result of  Dai-Fu \cite{DF}. When $k=1,2$, Theorem A,B,C are the result of Chen-Wang \cite{CW1,CW}. We note that a sign error in \cite[(4.9)]{DF} led to an incorrect Bochner formula for dimensions $4$ and $5$. This has now been corrected in (\ref{Bochner}).
 Hence when $n=4,5$, we can only choose $k=1$.
	\end{remark}
	Based on Lemma $3.3$, we present a rigidity theorem for Einstein $4$-manifolds in the case $k=3$, which improves  \cite[Theorem 1.8]{Li5} for $n=4$ and establishes its optimality in this case.
	\begin{maintheorem}\label{thm:main-theorem-4}
		Let $(M,g)$ be an Einstein four-manifold whose the curvature operator of the second kind satisfies 
		\begin{equation}\label{cone}
			\lambda_{1}+\lambda_{2}+\lambda_{3}\ge -3\bar{\lambda},
		\end{equation}
		then  one of the following must be true:\\
		(1) $M$ is flat.\\
		(2)  up to rescaling, $M$ is isometric to a quotient of the round sphere $(\mathbb{S}^4, g_0)$.\\
		(3)  up to rescaling, $M$ is isometric to a $\mathbb{CP}^2$ with the Fubini–Study metric.\\
		(4)  up to rescaling, $M$ is isometric to a quotient of $\mathbb{S}^2\times \mathbb{S}^2$ with the product metric.
	\end{maintheorem}
	\section{preliminaries}
	Here, we introduce our notation and summarize key facts regarding the curvature operator of the second kind. Further details can be found in \cite{Li4,Li5,NP,DF,CW}.
	
	Let $(V,g)$ be an $n$-dimensional Euclidean vector space and $\left\{ e_i \right\}_{i=1}^{n}$ be an orthonormal basis for $V$. We recall that the second-order tensor space can be decomposed by 
	\[T^2(V)={\Lambda }^2(V)\oplus S^2(V),\]
	where ${\Lambda}^2(V)$ and $S^2(V)$ is the second-order anti-symmetric tensor space and the second order symmetric tensor space on $V$, respectively. The space $S^2(V)$ can split subspaces as
	\[S^2(V)=S_{0}^2(V)\oplus \mathbb{R}g,\]
	where $S_{0}^{2}(V)$is the space of traceless symmetric two-tensors, and $g=\sum\limits_{i=1}^{n}{e_i\otimes e_i}$. We always denote by $N=\frac{(n-1)(n+2)}{2}$ the dimension of $S_{0}^{2}(V)$ in this paper.
	
	The space of symmetric two-tensors on ${{\Lambda }^{2}}(V)$ has the orthogonal decomposition
	\[S^2({\Lambda }^2(V))=S_{B}^2({\Lambda }^2(V))\oplus {\Lambda }^4(V),\]
	where $S_{B}^{2}({\Lambda }^2(V))$ consists of all tensors $R\in S^2({\Lambda }^2(V))$ that also satisfy the first Bianchi identity. $S_{B}^2({\Lambda }^2(V))$ is called the space of algebraic curvature operators.
	By the symmetries of $R\in S_{B}^{2}({{\Lambda }^{2}}(V))$, there are (up to sign) two natural ways in which induces a symmetric linear map from $T^2(V)$ to $T^2(V)$.
	The first one is defined by
	\begin{align*}
		& \hat{R}:{{\Lambda }^{2}}(V)\to {{\Lambda }^{2}}(V), \\ 
		& \hat{R}({{e}_{i}}\wedge {{e}_{j}})=\frac{1}{2}\sum\limits_{k,l}{{{R}_{ijkl}}{{e}_{k}}\wedge {{e}_{l}}}, 
	\end{align*}
	which is called the curvature operator of the first kind.
	The second one denoted by $\bar{R}$, is defined as
	\begin{align*}
		& \bar{R}:{{S}^{2}}(V)\to {{S}^{2}}(V), \\ 
		& \bar{R}({{e}_{i}}\odot {{e}_{j}})=\sum\limits_{k,l}{{{R}_{iklj}}{{e}_{k}}\odot {{e}_{l}}}, 
	\end{align*}
	where $e_i\odot e_j=e_i\otimes e_i+e_j\otimes e_i.$ $\bar{R}$ induces a symmetric form $\mathring{R}\,$ defined by
	\begin{align*}
		& \mathring{R}\,:S_{0}^{2}(V)\times S_{0}^{2}(V)\to R \\ 
		& ({{S}^{\alpha }},{{S}^{\beta }})\to {{R}_{ijkl}}S_{jk}^{\alpha }S_{il}^{\beta },  
	\end{align*}
	which is called curvature operator of the second kind. It was introduced by Nishikawa \cite{Ni} and as pointed out in \cite{NP}, the curvature operator of the second kind $\mathring{R}\,$ can also be interpreted as $\bar{R}$ and its image restrict on $S_{0}^{2}(V)$.

	Let $(M,g)$ be an $n$-dimensional Riemannian manifold.  For each $p\in M,$ let $V = T_pM$ and $R$ is denoted by the Riemannian curvature tensor of $M$. $\mathring{R}\,$ is denoted by curvature operator of second kind of $M$. We introduce the following definition which comes from \cite{NP}
	\begin{definition}
		Let $T^{(0,k)}(V)$denote the space of $(0,k)$-tensor space on $V$. For $S\in S^2(V)$ and $T\in T^{(0,k)}(V),$ we define
		\begin{align*}
			& S:T^{(0,k)}(V)\to T^{(0,k)}(V) \\ 
			& (ST)(X_1,\cdots ,X_k)=\sum\limits_{i=1}^{k}{T(X_1,\cdots ,SX_i,\cdots ,X_k)} 
		\end{align*}
		and define $T^{S^2}\in T^{(0,k)}(V)\otimes S^2(V)$ by
		$$\left\langle T^{S^2}(X_1,\cdots ,X_k),S \right\rangle =(ST)(X_1,\cdots ,X_k).$$
	\end{definition} 
	Hence, if $\left\{ {\tilde{S}}^i \right\}_{i=1}^{\frac{n(n+1)}{2}}$ is an orthonormal basis for $S^2(V)$, then
	\[T^{S^2}=\sum\limits_{i=1}^{\frac{n(n+1)}{2}}{{\tilde{S}}^iT\otimes {\tilde{S}}^i}.\]
	Similarly, we define $T^{S_{0}^2}\in T^{(0,k)}(V)\otimes S_{0}^{2}(V)$ and for an orthonormal basis $\left\{ S^i \right\}$ of $S_{0}^2(V)$
	$$T^{S_{0}^2}=\sum\limits_{i}{S^iT\otimes S^i}.$$

	Let $\{{\lambda }_i\}$ denote the eigenvalues of the operator $\mathring{R}\,$  and let $\left\{ S^i \right\}$ be the corresponding orthonormal eigenbasis of $S_{0}^2(V)$. In what follows, the eigenvalues $\{{\lambda}_i\}$ are assumed to be arranged in non‑decreasing order. Define the action of $\mathring{R}\,$ on the space $T^{S_{0}^{2}}$ as
	\[\mathring{R}\,(T^{S_{0}^{2}})=\sum\limits_{i}{S^iT\otimes \mathring{R}\,(S^i),}\]
	then
	\[\left\langle \mathring{R}\,(T^{S_{0}^{2}}),T^{S_{0}^{2}} \right\rangle =\sum\limits_{i}{{\lambda }_i\left| S^iT \right|^2}.\]
	
	Let $W$ be the Weyl curvature tensor of $M$, Set $S=\sum\limits_{i=1}^{N}{{{\left| {{S}^{i}}W \right|}^{2}}}$, $M =\underset{1\le i\le N}{\mathop{\max }}\,{{\left| {{S}^{i}}W \right|}^{2}}.$	The following estimate holds for the cone condition.
	\begin{proposition} \label{proposition 1}
		Let ${\lambda }_1\le {\lambda }_2\le \cdots \le {\lambda }_N$ with $\sum\limits_{i=1}^{N}{{\lambda }_{i}}=N\bar{\lambda }.$
		Under the cone condition $k^{-1}({\lambda }_1+\cdots +{\lambda }_k)\ge -\theta \bar{\lambda }$ for $0<k\le [\frac{S}{M}]$,
		we have
		\[\sum\limits_{i=1}^{N}{{\lambda }_i\left| S^{i}W \right|^2}\ge -S\theta \bar{\lambda }.\]
	\end{proposition}
	\begin{proof} First we have
		\begin{eqnarray}
			\sum\limits_{i=1}^{N}{{\lambda }_i\left| S^{i}W \right|^2}&\ge& \sum\limits_{i=1}^{k}{{\lambda }_i\left| S^{i}W \right|^2}+\sum\limits_{i=k+1}^{N}{{\lambda }_{k+1}\left| S^{i}W \right|^2} \nonumber\\ 
			&=&{{\lambda }_{k+1}}\sum\limits_{i=1}^{N}{\left| S^{i}W \right|^2}+\sum\limits_{i=1}^{k}{({\lambda }_i-{\lambda }_{k+1})\left| S^{i}W \right|^2} \nonumber\\ 
			&\ge& {\lambda }_{k+1}S+M \sum\limits_{i=1}^{k}{({\lambda }_i-{\lambda }_{k+1})} \nonumber\\ 
			&=&(S-kM ){\lambda }_{k+1}+M \sum\limits_{i=1}^{k}{\lambda }_i, \label{2.1}
		\end{eqnarray}
		If $k^{-1}({\lambda }_1+\cdots +{\lambda }_k)\ge -\theta \bar{\lambda }$, then 
		\begin{eqnarray}
			{{\lambda }_{k+1}}\ge \frac{{\lambda }_1+\cdots +{\lambda }_k}{k}\ge -\theta \bar{\lambda }.\label{2.2}
		\end{eqnarray} 
		Applying inequality (\ref{2.2}) to the right-hand side of (\ref{2.1}), we get for $0<k\le [\frac{S}{M}],$
		\begin{eqnarray*} 
			\sum\limits_{i=1}^{N}{{\lambda }_i\left| S^{i}W \right|^2}&\ge& (S-kM ){\lambda }_{k+1}+M \sum\limits_{i=1}^{k}{\lambda }_i \\ 
			&\ge& -\theta\bar{\lambda }(S-kM )-M k\theta \bar{\lambda } \\ 
			&=&-S\theta \bar{\lambda }. \\ 
		\end{eqnarray*}
	\end{proof}
	\section{$n$-dimensional Einstein manifolds}
	Let $(M^n ,g)$ be an $n$-dimensional Einstein manifold. Denote by $\{\lambda_{i}\}_{i=1}^{N}$ the eigenvalues of the operator   $\mathring{R}\,$ and let $\bar{\lambda}=\sum\nolimits_{i=1}^{N}{{\lambda }_i}/{N}\;$ be their average. From \cite{Li4}, it follows that $$tr(\mathring{R})=\frac{n+2}{2n}s,$$
	where $s$ is the scalar curvature of $M$. Consequently $\bar{\lambda}=\frac{s}{n(n-1)}$.
	
	Based on the calculations presented in \cite{DF}, it has been calculated that
	\begin{eqnarray}
		S=\sum\limits_{i=1}^{N}{{{\left| {{S}^{i}}W \right|}^{2}}}=\frac{2({{n}^{2}}+n-8)}{n}{{\left| W \right|}^{2}}=\frac{4N-12}{n}{{\left| W \right|}^{2}},\label{1}
	\end{eqnarray}
	and $$M =\underset{1\le i\le N}{\mathop{\max }}\,{\left| S^{i}W \right|^2}=\frac{8n-16}{n}{{\left| W \right|}^{2}}.$$ 
	Hence 
	$$[\frac{S}{M}]\ge [\frac{n+2}{4}].$$
	Moreover, on Einstein manifold, the square norm of the Weyl curvature can be expressed as 
	\begin{eqnarray}
		{{\left| W \right|}^{2}}=\frac{4}{3}\sum\limits_{i=1}^{N}{\lambda _{i}^{2}}-\frac{4N}{3}{{\bar{\lambda }}^{2}}.\label{2}
	\end{eqnarray}
	
	With these preparations we obtain the following estimate for the Bochner formula.
	\begin{lemma} \label{lemma3}
		Let $(M^n,g)$ be an Einstein manifold of dimension $n \ge 6$. If $k^{-1}({\lambda }_1+\cdots +{\lambda }_k)\ge -\theta \bar{\lambda }$ for $ k\le [\frac{n+2}{4}], \theta \ge 0$, then the curvature operator $R$ of $M$ satisfies
		\begin{eqnarray}
			3\left\langle \Delta R,R \right\rangle &\ge& 16\sum\limits_{i=1}^{N}{\lambda _{i}^3}+\frac{16}{3n}[2N-12n+6-(N-3)\theta ]\bar{\lambda }\sum\limits_{i=1}^{N}{\lambda _{i}^2} \nonumber\\ 
			&& -\frac{16N}{3n}[2N-9n+6-(N-3)\theta ]\bar{\lambda }^3.\label{3} 
		\end{eqnarray}
	\end{lemma}
	\begin{proof}
		The following Bochner formula comes from Proposition 3.4 in \cite{DF}.
		\begin{align*}
			3\left\langle \Delta R,R \right\rangle =&\sum\limits_{i}{{\lambda }_i\left| S^{i}W \right|^2}+8\left( \frac{-n^3+6n^2+12n-8}{3n^4(n-1)^2} \right)s^3 \\ 
			& +8\left( \frac{2n^2-22n+8}{3n^2(n-1)} \right)s\sum\limits_{i}{\lambda _{i}^2}+16\sum\limits_{i}{\lambda _{i}^3}. 
		\end{align*}
		Substituting $N=\frac{(n-1)(n+2)}{2}$ and $\bar{\lambda }=\frac{s}{n(n-1)}$, we simplify the equation to
		\begin{eqnarray*}
			3\left\langle \Delta R,R \right\rangle &=&\sum\limits_{i=1}^{N}{{{\lambda }_{i}}{{\left| {{S}^{i}}W \right|}^{2}}}-\frac{16N(2N-9n+6)}{3n}{{{\bar{\lambda }}}^{3}} \\ 
			&&+\frac{16(2N-12n+6)}{3n}\bar{\lambda }\sum\limits_{i=1}^{N}{\lambda _{i}^{2}}+16\sum\limits_{i=1}^{N}{\lambda _{i}^{3}}. \\ 
		\end{eqnarray*}
		
		Applying Proposition~\ref{proposition 1}, together with (\ref{1}) and (\ref{2}) yields
		\[\begin{aligned}
			\sum\limits_{i=1}^{N}{{\lambda }_i\left| S^{i}W \right|^2}&\ge -\frac{4N-12}{n}\theta \bar{\lambda }\left| W \right|^2 \\ 
			& =-\frac{16(N-3)}{3n}\theta \bar{\lambda }\sum\limits_{i=1}^{N}{{\lambda}_{i}^2}+\frac{16N(N-3)}{3n}\theta \bar{\lambda }^3. \\ 
		\end{aligned}\]
		Combining the above two equations yields the desired result. 
	\end{proof} 
	\begin{remark}
		When $k=2$, this is Lemma 3.3 of \cite{CW}.
	\end{remark}
	To establish the main theorem, we present the following algebraic results, which play a crucial role in the proof.
	\begin{lemma}\label{lemma1}
		Let $0\le {\lambda }_1\le {\lambda }_2\le \cdots \le {\lambda }_N$ and $\sum\limits_{i=1}^{N}{{\lambda }_i}=C>0$, where $N\geq3$. If $C=\frac{N(N-1)}{2N-1},$ then the  function
		\[F(\lambda _{1},\lambda _{2},\ldots,\lambda _{N})=\sum\limits_{i=1}^{N}{\lambda _{i}^3}-\sum\limits_{i=1}^{N}{\lambda _{i}^2}\]
		attains its minimal at $\left( \frac{C}{N},\frac{C}{N},\ldots ,\frac{C}{N} \right)$ and $\left( 0,\frac{C}{N-1},\ldots ,\frac{C}{N-1} \right).$
	\end{lemma}
	\begin{proof} 
		We will divide the proof into two cases.\\
		\textbf{Case 1.} Consider ${\lambda }_1>0$. Under the constraints $\sum\limits_{i}{{{\lambda }_{i}}}=C>0$, Lagrange multipliers yields a constant $\mu \in \mathbb{R}$ such that
		\begin{eqnarray}
			3\lambda _{i}^2-2{\lambda }_i+\mu =0\label{lambda1}
		\end{eqnarray} 
		for $i=1,\ldots ,N.$
		The possible solution of this equation is 
		\[a=\frac{1-\sqrt{1-3\mu }}{3},b=\frac{1+\sqrt{1-3\mu }}{3}.\]
		Hence, the possible critical points are 
		\[Q_k=\left( \overset{k}{\mathop{\overset\frown{a,\ldots ,a}}}\,,\overset{N-k}{\mathop{\overset\frown{b,\ldots ,b}}}\, \right).\]
		Under constrain  $ka+(N-k)b=C$, we obtain
		$$(N-2k)\sqrt{1-3\mu }=3C-N.$$
		Because $C=\frac{N(N-1)}{2N-1}$ implies $3C-N>0$, we must have $N-2k>0$. Consequently, 
		\begin{eqnarray}
			\mu =\frac{1}{3}[1-{\left( \frac{3C-N}{N-2k} \right)^2}].\label{lambda2}
		\end{eqnarray} 
		
		Substituting $Q_k$ into $F$ and applying the relation (\ref{lambda1}) to obtain
		\begin{align*}
			F(Q_k)&=-\frac{C}{3}\mu -\frac{1}{3}\sum\limits_{i}{\lambda _{i}^{2}} \\ 
			& =-\frac{C}{3}\mu -\frac{N(2-3\mu )+2(N-2k)\sqrt{1-3\mu }}{27}. \\ 
		\end{align*}
		Together with (\ref{lambda2}) given
		\begin{align*}
			F(Q_k)&=-\frac{C}{3}\frac{1}{3}[1-{{\left( \frac{3C-N}{N-2k} \right)}^{2}}]-\frac{N}{27}-\frac{N{{\left( \frac{3C-N}{N-2k} \right)}^{2}}+2(3C-N)}{27} \\ 
			& =(\frac{C}{9}-\frac{N}{27}){{\left( \frac{3C-N}{N-2k} \right)}^{2}}-\frac{C}{3}+\frac{N}{27}. \\ 
		\end{align*}
		Since we assume $C=\frac{N(N-1)}{2N-1}>\frac{N}{3},$ it follow that $\frac{C}{9}-\frac{N}{27}>0$. Hence
		\begin{align*}
			& \min F(Q_k)=\underset{0\le k<{N}/{2}\;}{\mathop{\min }}\,\{(\frac{C}{9}-\frac{N}{27}){\left( \frac{3C-N}{N-2k} \right)^2}-\frac{C}{3}+\frac{N}{27}\} \\ 
			& =F(Q_0), \\ 
		\end{align*}
		and the minimal point is $\left( \frac{C}{N},\frac{C}{N},\ldots ,\frac{C}{N} \right).$\\\\
		\textbf{Case 2}, Consider ${\lambda }_1=0$.
		Suppose $X$ be a minimizer such that it has the maximum number of zeros among all minimizers. We claim that the non-zero elements of $X$ are all equal.
		
		Assume, for a contradiction, that ${\lambda }_i>{\lambda }_j>0$ for some $i,j$. If ${\lambda }_i+{\lambda }_j\le \frac{2}{3},$ we have
		\begin{align*}
			& \lambda _{i}^3-\lambda _{i}^2+\lambda _{j}^3-\lambda _{j}^2=({\lambda }_i+{\lambda }_j)^3-({\lambda }_i+{\lambda }_j)^2+{\lambda }_i{\lambda }_j(2-3{\lambda }_i-3{\lambda }_j) \\ 
			& \ge ({\lambda }_i+{\lambda }_j)^3-({\lambda }_i+{\lambda }_j)^2.
		\end{align*}
		Thus, replacing ${\lambda }_i,{\lambda }_j$ with ${\lambda }_i+{\lambda }_j,0$ yields another minimizer. This contradicts the fact that $X$ has the largest number of zeros.
		
		If  ${\lambda }_i+{\lambda }_j>\frac{2}{3},$ letting $u=\frac{{\lambda }_i+{\lambda }_j}{2}$, we have
		\begin{align*}
			& \lambda _{i}^3-\lambda _{i}^2+\lambda _{j}^3-\lambda _{j}^2=2(u^3-u^2)+\frac{1}{4}(3{\lambda }_i+3{\lambda }_j-2)({\lambda }_i-{\lambda }_j)^2 \\ 
			& >2(u^3-u^2). 
		\end{align*}
		Thus replacing ${\lambda }_i,{\lambda }_j$ with $\frac{{\lambda }_i+{\lambda }_j}{2}$, $\frac{{\lambda }_i+{\lambda }_j}{2}$, yields a smaller  minimizer, this contradicts the minimality of $X$. The claim is proved.
		
		Through the above discussion, we can assume that for $X$, $0={\lambda }_1=\cdots ={\lambda }_k,$ and ${\lambda }_{k+1}=\cdots ={\lambda }_N$ for some $1\le k< N,$ then
		\[F(X)=\frac{C^3}{(N-k)^2}-\frac{C^2}{N-k}.\]
		So in $\partial \Omega,$
		$$\underset{\partial \Omega }{\mathop{\min }}\,F=\underset{1\le k< N}{\mathop{\min }}\,\{\frac{C^3}{(N-k)^2} -\frac{C^2}{N-k}\}.$$
		Since we assume $C=\frac{N(N-1)}{2N-1},$ we obtain 
		$$\frac{C^3}{(N-k)^2}-\frac{C^2}{N-k}-(\frac{C^3}{N^2}-\frac{C^2}{N})=\frac{{C^2}(N^2-(N-k)^2)(C-\frac{N(N-k)}{2N-k})}{(N-k)^2N^2}\ge 0,$$
		equality holds if and only if $k=1$. Hence,
		$$\underset{\partial \Omega }{\mathop{\min }}\,F=\underset{1\le k< N}{\mathop{\min }}\,\{\frac{C^3}{(N-k)^2}-\frac{C^2}{N-k}\}=\frac{C^3}{N^2}-\frac{C^2}{N},$$
		and the minimal point is $\left( 0,\frac{C}{N-1},\ldots ,\frac{C}{N-1} \right)$. As $X$ possesses at least one zero, it follows that the minimizer must be unique.
		
		Combining case 1 and case 2, we have completed the proof of the Lemma.
	\end{proof}
	\begin{remark}
		From case 2, we can conclude that if $C>\frac{N(N-1)}{2N-1}$, then the only minimal point of $F$ is $\left( \frac{C}{N},\frac{C}{N},\ldots ,\frac{C}{N} \right)$.
	\end{remark}

	When $k^{-1}({\lambda }_1+{\lambda }_2+\cdots +{\lambda }_k)\ge -\theta \bar{\lambda }$ for some constant $\theta$, the following estimate holds for ${\lambda }_1$.		
	\begin{lemma} \label{lemma2}
		Let ${\lambda }_1\le {\lambda }_2\le \cdots \le {\lambda }_N$ be a sequence with ${\lambda }_1+{\lambda }_2+\cdots +{\lambda }_N=N\bar{\lambda }$. If there is constant $\theta$ such that $k^{-1}({\lambda }_1+{\lambda }_2+\cdots +{\lambda }_k)\ge -\theta \bar{\lambda}$ for some $k< N$, then
		\[{{\lambda }_{1}}\ge -\frac{(N-1)k\theta +N(k-1)}{N-k}\bar{\lambda }.\]
	\end{lemma}
	\begin{proof}
		
		Since ${{\lambda }_{i}}$ is a non-decreasing sequence, and $k^{-1}({\lambda }_1+{\lambda }_2+\cdots +{\lambda }_k)\ge -\theta \bar{\lambda}$, 
		$${{\lambda }_{i_1}}+{{\lambda }_{i_2}}+\cdots +{{\lambda }_{i_k}}\ge -k\theta \bar{\lambda} \quad\text{for any}\quad 1\le i_1<i_2<\cdots <i_k\le N.$$
		Among all such summation sequences, there are $\left( \begin{aligned}
			& N-1 \\ 
			& k-1 \\
		\end{aligned} \right)$ of them that contain ${\lambda }_1$, summing over these sequences obtain
		\[\left( \begin{aligned}
			& N-1 \\ 
			& k-1 \\ 
		\end{aligned} \right){{\lambda }_{1}}+\left( \begin{aligned}
			& N-2 \\ 
			& k-2 \\ 
		\end{aligned} \right)\sum\limits_{i=2}^{N}{{\lambda }_i}\ge -\left( \begin{aligned}
			& N-1 \\ 
			& k-1 \\ 
		\end{aligned} \right)k\theta \bar{\lambda}\]
		since there are $\left( \begin{aligned}
			& N-2 \\ 
			& k-2 \\ 
		\end{aligned} \right)$ sequences that include ${\lambda }_1$ and ${\lambda }_i,i\ne 1.$ For ${{\lambda }_{1}}+{{\lambda }_{2}}+\cdots +{{\lambda }_{N}}=N\bar{\lambda}$, we finally get 
		\[\left[ \left( \begin{aligned}
			& N-1 \\ 
			& k-1 \\ 
		\end{aligned} \right)-\left( \begin{aligned}
			& N-2 \\ 
			& k-2 \\ 
		\end{aligned} \right) \right]{{\lambda }_{1}}\ge -\left[ \left( \begin{aligned}
			& N-1 \\ 
			& k-1 \\ 
		\end{aligned} \right)k\theta +N\left( \begin{aligned}
			& N-2 \\ 
			& k-2 \\ 
		\end{aligned} \right) \right]\bar{\lambda},\]
		i.e. ${{\lambda }_{1}}\ge -\frac{(N-1)k\theta +N(k-1)}{N-k}\bar{\lambda }.$
	\end{proof}
	\begin{proof}[\textbf{Proof of Theorem~\ref{thm:main-theorem-1}}]According to \cite[Proposition 2.8]{Li5}, we have  $\bar{\lambda}\geq0$ for $\theta(n,k)>0$. If $\bar{\lambda}=0$, then $M$ is flat. Now we consider the case $\bar{\lambda}> 0$, i.e., $s>0$. Hence $M$ is compact by Myers Theorem.
		
		Writing the right term of (\ref{3}) as a function $F(k,\theta ,\lambda )$, 
		$$F(k,\theta ,\lambda )=16\sum\limits_{i=1}^{N}{\lambda _{i}^{3}}+\frac{16}{3n}(2N-12n+6-(N-3)\theta )\bar{\lambda }\sum\limits_{i=1}^{N}{\lambda _{i}^2} -\frac{16N}{3n}(2N-9n+6-(N-3)\theta )\bar{\lambda }^3.$$
		Now we consider the minimal of $F$ under constraint \[{\lambda }_1\le {\lambda }_2\le \cdots \le {\lambda }_N,\quad k^{-1}({\lambda }_1+\cdots +{\lambda }_k)\ge -\theta \bar{\lambda } \quad\text{and}\quad \sum\limits_{i=1}^{N}{{\lambda }_i}=N\bar{\lambda }.\]
		Obviously, $F(k,\theta ,\bar{\lambda})=0$ when ${\lambda }_1= {\lambda }_2= \cdots  {\lambda }_N=\bar{\lambda}$, since in this case $M$ has constant curvature and the estimate in Lemma~\ref{lemma3} becomes trivial.
		
		Set
		$${\beta }_i={\lambda }_i+D\bar{\lambda }, \quad\text{where}\quad D=\frac{(N-1)k\theta +N(k-1)}{N-k}.$$ 
		From Lemma~\ref{lemma2} we know 
		$${\beta }_i\ge 0 \quad\text{and}\quad \sum\limits_{i}{{\beta }_i}=N(1+D)\bar{\lambda },$$
		then $F(k,\theta ,\lambda )$ will become
		\begin{align*}
			F(k,\theta ,\beta )=&16\sum\limits_{i=1}^{N}{({\beta }_i-D\bar{\lambda })^3}+\frac{16}{3n}(2N-12n+6-(N-3)\theta )\bar{\lambda }\sum\limits_{i=1}^{N}{({\beta }_i-D\bar{\lambda })^2} \\ 
			& -\frac{16N}{3n}(2N-9n+6-(N-3)\theta){\bar{\lambda }}^3 \\ 
			=&16\left\{\sum\limits_{i=1}^{N}{\beta _{i}^3}+[\frac{1}{3n}(2N-12n+6-(N-3)\theta )-3D]\bar{\lambda }\sum\limits_{i=1}^{N}{\beta _{i}^2}\right\}+G(\theta ,k){\bar{\lambda }}^3. \\ 
		\end{align*}
		where $G(\theta,k)$ is a polynomial in terms of $\theta$ and $k$ and it is independent of the variables ${\beta}_i$. When $F$ is regarded as a function of the the term 
		${\beta}_i$, $G(\theta,k)$ is therefore a constant and can be omitted specific expressions, as it does not influence the minimizer of $F$.
		
		We choose $\theta $ such that 
		\[\frac{1}{3n}(2N-12n+6-(N-3)\theta )-3D=-\frac{2N-1}{N-1}(1+D).\]
		Substituting $D=\frac{(N-1)k\theta +N(k-1)}{N-k}$ into the above equation, we obtain the solution:
		\begin{eqnarray}\label{theta}
			\theta =\frac{(N-k)(2N-9n+6)-3nN(k-2)}{(N-3)(N-k)+3kn(N-2)}.
		\end{eqnarray}
		Then \begin{align*}
			F(k,\theta ,\beta )=&16\left\{\sum\limits_{i=1}^{N}{\beta _{i}^3}-\frac{2N-1}{N-1}(1+D)\bar{\lambda }\sum\limits_{i=1}^{N}{\beta _{i}^2}\right\}+G(\theta ,k){{{\bar{\lambda }}}^3} \\ 
			=&\frac{16(2N-1)^3(1+D)^3}{(N-1)^3}\bar{\lambda }^3\left\{\sum\limits_{i=1}^{N}{\left( \frac{N-1}{(2N-1)(1+D)\bar{\lambda }}{\beta }_i \right)^3}-\sum\limits_{i=1}^{N}{\left( \frac{N-1}{(2N-1)(1+D)\bar{\lambda }}{\beta }_i \right)^2}\right\}\\
			&+G(\theta ,k){\bar{\lambda }}^3. 
		\end{align*}
		
		Since 
		\[\sum\limits_{i}{\frac{N-1}{(2N-1)(1+D)\bar{\lambda }}{{\beta }_{i}}}=\frac{N-1}{(2N-1)(1+D)\bar{\lambda }}\sum\limits_{i}{{{\beta }_{i}}}=\frac{N(N-1)}{2N-1},\]
		by Lemma~\ref{lemma1} we know the minimal point of $F(k,\theta ,\beta )$ is 
		$$(\beta_1,\beta_2,\ldots,\beta_N)=\left( (1+D)\bar{\lambda },(1+D)\bar{\lambda },\ldots ,(1+D)\bar{\lambda } \right)$$
		and $$(\beta_1,\beta_2,\ldots,\beta_N)=\left( 0,\frac{N(1+D)}{N-1}\bar{\lambda },\ldots ,\frac{N(1+D)}{N-1}\bar{\lambda } \right).$$
		i.e.
		\begin{eqnarray}
			(\lambda_1,\lambda_2,\ldots,\lambda_N)=(\bar{\lambda },\bar{\lambda },\ldots \bar{\lambda }) \label{point1}
		\end{eqnarray}
		and
		\begin{eqnarray}
			(\lambda_1,\lambda_2,\ldots,\lambda_N)=(-D\bar{\lambda },\frac{N+D}{N-1}\bar{\lambda },\ldots ,\frac{N+D}{N-1}\bar{\lambda })\label{point2}.
		\end{eqnarray}
		Hence, the minimal value of $F$ is $F(k,\theta ,\bar{\lambda })=0.$
		
		By Lemma~\ref{lemma3}, we have
		$$\frac{1}{2}\Delta {{\left| R \right|}^{2}}={{\left| \nabla R \right|}^{2}}+\left\langle \Delta R,R \right\rangle \ge 0,$$
		thus $$\nabla R=0,$$
		and $$\left\langle \Delta R,R \right\rangle =0.$$
		Therefore, $M$ is a locally symmetric space. Since when $n\ge 8,$ $\theta(n,k)$ defined by (\ref{theta}) satisfies $\theta < \frac{2(n-1)}{n+2}$ for $1\le k\le \left[ \frac{n+2}{4} \right]$. It follows from Theorem 1.8 and Proposition 2.9 of \cite{Li5} that $M$ is either flat (when $\bar{\lambda }=0$) or a rational homology sphere (when $\bar{\lambda }>0$).
		
		The complete classification of compact symmetric spaces that are rational homology spheres is due to Wolf \cite[Theorem 1]{Wo}. Other than spheres, the only simply-connected example is the space ${SU(3)}/{SO(3)}.$ (\ref{point1}) and (\ref{point2}) implies that the eigenvalues of the curvature operator are either 
		$$(\bar{\lambda },\bar{\lambda },\ldots \bar{\lambda })$$
		or 
		$$(-D\bar{\lambda },\frac{N+D}{N-1}\bar{\lambda },\ldots ,\frac{N+D}{N-1}\bar{\lambda }),$$
		where $D=\frac{(N-1)k\theta +N(k-1)}{N-k}$.  In conjunction with the result of \cite[Example 4.5]{NP}, it follows that the manifold must be a spherical space form,
		which completes the proof of Theorem~\ref{thm:main-theorem-1}
	\end{proof}
	
	\begin{remark}
		When $k\le [\frac{n+2}{4}]$, the expression of $\theta$ in (\ref{theta}) must be non-negative, that is 
		$$k\le \frac{2N(N+3)-3nN}{3n(N-3)+2N+6}\quad \text{when}\quad k\le [\frac{n+2}{4}].$$ 
		When $n\ge 14$, we have 
		$$\frac{2N(N+3)-3nN}{3n(N-3)+2N+6}>\frac{n+2}{4}.$$ 
		When $4 \le n \le 13$, 
		$$[\frac{n+2}{4}]<\frac{2N(N+3)-3nN}{3n(N-3)+2N+6}<\frac{n+2}{4},$$
		except $n=6,7,10$. When $n=6,7$, $k$ can choose $1$ and $n=10$, $k$ can choose $1,2.$		
	\end{remark}
	\begin{proof}[\textbf{Proof of Theorem~\ref{thm:main-theorem-2}}] The corrected version of \cite[(4.9)]{DF} is as follows.
\begin{equation*}
		\begin{aligned}
			{R_{tspq}}{R_{tjpl}}{R_{sjql}}
						=& {W_{tspq}}{W_{tjpl}}{W_{sjql}} - \frac{3s}{2n(n - 1)}\left| W \right|^2 + \frac{(n-2)s^3}{n^2(n - 1)^2}\\
			=& {W_{tspq}}{W_{tjpl}}{W_{sjql}} - \frac{3s}{2n(n - 1)}\left| R \right|^2 + \frac{(n+1)s^3}{n^2(n-1)^2}.\\
		\end{aligned}
	\end{equation*}
Here the sign for $\frac{{3s}}{{2n(n - 1)}}{\left| W \right|^2}$ is given as `$+$' in \cite[(4.9)]{DF}, but it should be corrected to `$-$'.
Following the approach of \cite{DF},  we present the corrected formula as follows:
\begin{align}\label{Bochner}
		& \left\langle \Delta R,R \right\rangle =8\sum\limits_{i }{\lambda _{i}^{3}}+\frac{8(n-10)}{3n(n-1)}s\sum\limits_{i}{\lambda _{i}^2}-\frac{4(n-7)(n+2)}{3n^3(n-1)^2}s^3 \nonumber\\ 
		=&8\sum\limits_{i}{\lambda _{i}^3}+\frac{8(n-10)}{3}\bar{\lambda }\sum\limits_{i}{\lambda _{i}^{2}}-\frac{4(n+2)(n-7)(n-1)}{3}\bar{\lambda }^3.
	\end{align}

	We use the above Bochner formula, and the same methods as the proof of Theorem A, then take $n=4,5$ and $k=1$. We get $n=4$, $\theta=\frac{1}{7}$ and $n=5,\theta=\frac{4}{9}$. The remainder of the argument is analogous to that in Theorem~\ref{thm:main-theorem-1}, which completes the proof of Theorem~\ref{thm:main-theorem-2}.
\end{proof}
	\begin{proof}[\textbf{Proof of Theorem~\ref{thm:main-theorem-3}}]
		The argument follows the same pattern as in Theorem ~\ref{thm:main-theorem-1}, with (\ref{theta}) applied for the cases where $n=6,7$ with $k=1$, and $n=10$ with $k=1,2$. This completes the proof of Theorem~\ref{thm:main-theorem-3}.
	\end{proof}
	
	\section{$4$-dimensional Einstein manifolds}
	In this section, we present a geometric application of Lemma~\ref{lemma1} in $4$-dimensional Einstein manifolds.
	
	\begin{proof}[\textbf{Proof of Theorem~\ref{thm:main-theorem-4}}]
		The  Bochner formula (\ref{Bochner}) on $4$-dimensional Einstein manifolds can be written as
\begin{equation}\label{Bochner2}
		\left\langle \Delta R,R \right\rangle =8\sum\limits_{\alpha }{\lambda _{\alpha }^{3}}-\frac{4}{3}s\sum\limits_{\alpha }{\lambda _{\alpha}^2}+\frac{1}{24}s^3 .	\end{equation}
		
		Denote the eigenvalues of $W^+$ by $a_{1} \le a_{2} \le a_{3}$, and those of $W^-$ by $b_{1} \le b_{2} \le b_{3}$, then
\begin{equation}\label{Weyl}
a_{1}+a_{2}+ a_{3}=0, b_{1}+b_{2}+b_{3}=0.
\end{equation}		
		According to \cite{CGT}, on Einstein 4-manifold, the eigenvalues of $\mathring{R}$ are given by
		\begin{equation}\label{RW}
			{\lambda}_{ij}=\frac{s}{12}-a_i-b_j,\quad i,j=1,2,3.
		\end{equation}		
		If  $\mathring{R}$ satisfies (\ref{cone}), then
		\[a_3=\frac{s}{12}-\frac{\lambda_{31}+\lambda_{32}+\lambda_{33}}{3}\le \frac{s}{12}+\bar{\lambda}=\frac{s}{6},\]
		and \[b_3=\frac{s}{12}-\frac{\lambda_{13}+\lambda_{23}+\lambda_{33}}{3}\le \frac{s}{12}+\bar{\lambda}=\frac{s}{6}.\]
		Consequently, $s\ge0$;  moreover,  if $s=0$, then $W^{\pm}=0$, i.e., $M$ is flat. 
		
      By substituting (\ref{RW}) into (\ref{Bochner2}) and applying (\ref{Weyl}), we directly compute to obtain
      \begin{equation}\label{Bochner3}
		\left\langle \Delta R,R \right\rangle =2[s(a_{1}^2+a_{2}^2+a_{3}^2)-12(a_{1}^3+a_{2}^3+a_{3}^3)]+2[s(b_{1}^2+b_{2}^2+b_{3}^2)-12(b_{1}^3+b_{2}^3+b_{3}^3)] .	\end{equation}
      Let
		$$f_1(a_1,a_2,a_3)=s(a_{1}^2+a_{2}^2+a_{3}^2)-12(a_{1}^3+a_{2}^3+a_{3}^3), f_2(b_1,b_2,b_3)=s(b_{1}^2+b_{2}^2+b_{3}^2)-12(b_{1}^3+b_{2}^3+b_{3}^3).$$ 		
		Now we consider the case $s>0$. The expression for $f_1$ can be rewritten as
		\begin{align*}
			& f_1=s(a_{1}^2+a_{2}^2+a_{3}^2)-12(a_{1}^3+a_{2}^3+a_{3}^3) \\ 
			& =s\sum\limits_{i=1}^{3}{(\frac{s}{6}-a_i-\frac{s}{6})^2}+12\sum\limits_{i=1}^{3}{(\frac{s}{6}-a_i-\frac{s}{6})^3} \\ 
			& =12\sum\limits_{i=1}^{3}{(\frac{s}{6}-a_i)^3}-5s\sum\limits_{i=1}^{3}{(\frac{s}{6}-a_i)^2}+\frac{s^3}{4} \\ 
			& =\frac{125}{144}s^3\left[ \sum\limits_{i=1}^{3}{[\frac{12}{5s}(\frac{s}{6}-a_i)]^3}-\sum\limits_{i=1}^{3}{[\frac{12}{5s}(\frac{s}{6}-a_i)]^2} \right]+\frac{s^3}{4}.
		\end{align*}
		Since 
		\[\frac{12}{5s}(\frac{s}{6}-a_i)\ge 0\quad\text{and}\quad \sum\limits_{i}{\frac{12}{5s}(\frac{s}{6}-a_i)}=\frac{6}{5},\]
		it follows from Lemma~\ref{lemma1} that the minimal points of $f_1$ are 
		$$(\frac{12}{5s}(\frac{s}{6}-a_1),\frac{12}{5s}(\frac{s}{6}-a_2),\frac{12}{5s}(\frac{s}{6}-a_3))=(\frac{2}{5},\frac{2}{5},\frac{2}{5})$$
		and  
		$$(\frac{12}{5s}(\frac{s}{6}-a_1),\frac{12}{5s}(\frac{s}{6}-a_2),\frac{12}{5s}(\frac{s}{6}-a_3))=(\frac{3}{5},\frac{3}{5},0),$$
		which correspond respectively to $$(a_1,a_2,a_3)=(0,0,0) \quad\text{and}\quad (a_1,a_2,a_3)=(-\frac{s}{12},-\frac{s}{12},\frac{s}{6}).$$
		Hence,
		$$f_1\ge f_1(0,0,0)=f_1(-\frac{s}{12},-\frac{s}{12},\frac{s}{6})=0.$$
Similarly, the minimal points of $f_2$ are $(b_1,b_2,b_3)=(0,0,0)$ and $(b_1,b_2,b_3)=(-\frac{s}{12},-\frac{s}{12},\frac{s}{6}),$  and
$$f_2\ge f_2(0,0,0)=f_2(-\frac{s}{12},-\frac{s}{12},\frac{s}{6})=0.$$
		In summary, we have
		$$\frac{1}{2}\Delta {{\left| R \right|}^{2}}={{\left| \nabla R \right|}^{2}}+\left\langle \Delta R,R \right\rangle \ge 0.$$
		By the maximum principle, we obtain $\nabla R=0$, and either\\ 
		(i)$(a_1,a_2,a_3)=(b_1,b_2,b_3)=(0,0,0),$ \\
		(ii)$(a_1,a_2,a_3)=(0,0,0)$ and $(b_1,b_2,b_3)=(-\frac{s}{12},-\frac{s}{12},\frac{s}{6}),$\\ 
		(iii)$(a_1,a_2,a_3)=(-\frac{s}{12},-\frac{s}{12},\frac{s}{6})$ and $(b_1,b_2,b_3)=(0,0,0)$,\\ 
		(iv)$(a_1,a_2,a_3)=(b_1,b_2,b_3)=(-\frac{s}{12},-\frac{s}{12},\frac{s}{6}).$ \\
	Therefore,  the curvature operator $\hat{R}$ of the first kind has eigenvalues $(\frac{s}{12},\frac{s}{12},\frac{s}{12},\frac{s}{12},\frac{s}{12},\frac{s}{12})$, $(\frac{s}{12},\frac{s}{12},\frac{s}{12},0,0,\frac{s}{4})$, $(0,0,\frac{s}{4},\frac{s}{12},\frac{s}{12},\frac{s}{12})$, or $(0,0,\frac{s}{4},0,0,\frac{s}{4})$.	
		Applying the classification of $4$-dimensional symmetric spaces, we complete the proof of Theorem D.
	\end{proof}

	\begin{corollary}
		Let $(M,g)$ be an Einstein four-manifold with $4\frac12$-nonnegative curvature operator of the second kind.
		Then  one of the following must be true:\\
		(1) $M$ is flat.\\
		(2)  up to rescaling, $M$ is isometric to a quotient of the round sphere $(\mathbb{S}^4, g_0)$.\\
		(3)  up to rescaling, $M$ is isometric to $\mathbb{CP}^2$ with the Fubini–Study metric.
	\end{corollary}
	\begin{remark}
		This corollary has already been proven in \cite{CGT} using a different method.
	\end{remark}
	\begin{proof}
		Suppose that ${\lambda }_1\le {\lambda }_2\le \cdots \le {\lambda }_{9}$
		are the eigenvalues of $\mathring{R}\,$ and $\bar{\lambda }$ is their average. Then $4\frac12$-nonnegative curvature operator of the second kind $\mathring{R}\,$ satisfies
		\begin{align*}
			& {\lambda }_1+{\lambda }_2+{\lambda }_3\geq-{\lambda }_4-\frac12 {\lambda }_5\\ 
			&\geq-\frac{\frac12 {\lambda }_5+{\lambda }_6+{\lambda }_7+{\lambda }_8+{\lambda }_9}{3} \\ 
			& \geq-\frac{{\lambda }_1+{\lambda }_2+{\lambda }_3+{\lambda }_4+{\lambda }_5+{\lambda }_6+{\lambda }_7+{\lambda }_8+{\lambda }_9}{3} \\ 
			& =-3\bar{\lambda }.
		\end{align*}
		Combining Theorem~\ref{thm:main-theorem-4}, we complete the proof  of  Corollary 4.1 because $\mathbb{S}^2\times \mathbb{S}^2$ has $4\frac12$-negative curvature operator of the second kind.
	\end{proof}
	

\end{document}